\documentclass[11pt,a4paper,oneside,onecolumn,fleqn]{article}
\usepackage{mathrsfs}
\usepackage{amsfonts}
\usepackage{longtable}
\usepackage{mathtools}
\usepackage{colortbl}
\usepackage{latexsym}
\usepackage{amsmath,amsthm}
\usepackage{epsf}
\usepackage{graphicx}
\usepackage{indentfirst}
\usepackage{cite}
\usepackage[numbers,sort&compress]{natbib}
\usepackage{caption2}

\renewcommand{\arraystretch}{1}

\theoremstyle{plain}
\newtheorem{thm}{\bf Theorem}[section]
\newtheorem{con}[thm]{\bf Construction}
\newtheorem{cor}[thm]{\bf Corollary}

\newtheorem{lem}[thm]{\bf Lemma}

\theoremstyle{definition}

\theoremstyle{remark}

\title{\bf A family of  multimagic squares based on large sets of orthogonal arrays
}
\author
 {   {  \small  \  Yong Zhang$^{1}$,  \ Kejun Chen$^{2}$ } \\
         {\footnotesize  {\it 1.  School of Mathematical and Statistics, Yancheng Teachers University, Jiangsu 224002, China}}\\
          \footnotesize { \it 2.  Department of Mathematics, Nanjing Normal University of Special Education, Jiangsu 210038,   China}\\
          \footnotesize { \it 3. Department of Computer Science, Lakehead University, Thunder Bay, ON P7B 5E1, Canada}\\
  }
\date{}
\begin{document}

\maketitle
{\baselineskip 15pt
\begin {abstract}

\noindent
   Large set of orthogonal arrays (LOA) were introduced by D. R. Stinson, and it is also used to construct multimagic squares recently. In this paper,  multimagic squares based on strong double LOA are further investigated. It is proved that
 there exists an MS$(q^{2t-1},t)$ for any prime power $q\geq 2t-1$ with $t\geq3$, which provided a new family of multimagic squares.
\end{abstract}

{\footnotesize\textbf{Keywords:} Multimagic magic square, Orthogonal array, Large set.

}}

{\begingroup\makeatletter\let\@makefnmark\relax\footnotetext{$^*$Corresponding to: zyyctu@gmail.com, yzhang63@lakeheadu.ca  (Zhang),  yctuckj@163.com(Chen)}
{\begingroup\makeatletter\let\@makefnmark\relax\footnotetext{Supported by the Natural Science Foundation of China: No.11301457(Zhang),  No.11371308(Chen). }


\vskip 0.5cm
\section{Introduction}

 An $n\times n$ matrix $A$ consisting of $n^2$   integers is a {\it general magic square of order $n$} if the sum of $n$ elements in each row, each column and each of two   diagonals are the same.  The sum is the \emph{magic number}.   A general magic square of order $n$ is  a \emph{magic square}, denoted by  MS$(n)$,  if the entries  are $n^2$ consecutive integers.  Usually, the $(i,j)$ entry of a matrix $A$ is denoted by $a_{i,j}$.    A lot of work has been done on   magic squares (\cite{Abe,Andrews,Handbook}).

Let $t$ be a positive  integer.  A   general magic square $M$   is a \emph{general $t$-multimagic square}  if  $M^{*1}$,$M^{*2}$, $\cdots$, $M^{*t}$ are all general magic squares, where  $M^{*e}=(m_{i,j}^{e})$, $e=1,2,\dots,t$. A  general $t$-multimagic square of order $n$ is  a \emph{$t$-multimagic square}, denoted by  MS($n,t$), if the entries  are $n^2$ consecutive integers.
Usually, a $2$-multimagic square is called a \emph{bimagic square} and a $3$-multimagic square   is called a \emph{trimagic squares}.

In 2007, Derksen, Eggermont and  van den Essen  \cite{Derksen} provided a constructive procedure to make a large class of $t$-multimagic squares for each positive integer $t\ge 2$.  There are also some new families of bimagic suqares and trimagic squares, we refer the readers to \cite{Chen2,Zhangcms,Zhang}. In \cite{Zhang},  large set of orthogonal arrays was used to construct multimagic squares, and it is also used in this paper.

An \emph{orthogonal array} of \emph{size} $N$, with $k$ \emph{constraints}, $v$ \emph{levels}, and \emph{strength} $t$, denoted by OA$(N; t, k, v)$, is a $k \times N$ array with entries from a set of $v \geq 2$ symbols, having the property that in every $t \times N$ submatrix, every $t \times 1$ column vector appears the same number $\frac{N}{v^t}$   times.
An OA$(N; t, k, v)$   is also denoted by OA$_{\lambda}(t,k,v)$, where $\lambda = \frac{N}{v^t}$. The parameter $\lambda$ is the
\emph{index} of the orthogonal array. An orthogonal array is {\it simple} if it does not contain two identical columns.
 As is well known (see \cite{Handbook}),  orthogonal arrays  are of importance in design theory.

A {\it large set} of orthogonal arrays OA$(N; t, k, v)$, denoted by  LOA$(N; t, k, v)$, is   a set of $v^k/N$
simple arrays OA$(N; t, k, v)$ such that every possible $k$-tuple of symbols occurs as a column in
exactly one of the OAs in the set. Equivalently, the union of the OAs forms a  trivial  OA$(v^k;  t, k, v)$. Large sets of orthogonal arrays   have been used to construct resilient functions and zigzag functions by D. R. Stinson (\cite{Stinson2,zigzag2}).
A special LOA, strong double LOA was introduced in \cite{Zhang}  to construct $t$-multimagic squares.

 Suppose that there is  an LOA$(N; t, k, v)$, $A_s=(a_{i,j}^{(s)}), s\in I_{ \frac{v^k}{N}}$.  Let  $B_j$ be an array whose $s$-th column is the same as the $j$-th  column of $A_s$, i.e.,
 \begin{center}
$\begin{aligned}
 \mbox{}\hspace{1.35in}B_j=(b_{i,s}^{(j)}),\    b_{i,s}^{(j)}=a_{i,j}^{(s)},\ i\in I_k,\ s\in I_{\frac{v^k}{N}}, \ j\in I_{N}, \mbox{}\hspace{1.35in}
\end{aligned}$
\end{center}
 \noindent Then $\{A_0,A_1,\cdots, A_{\frac{v^k}{N}-1}\}$ is called a {\it    double } LOA$(N; t, k, v)$, denoted by  DLOA $(\frac{v^k}{N},N$; $t, k, v)$, if $B_0, B_1, \cdots, B_{N-1}$   form an  LOA$(\frac{v^k}{N}; t, k, v)$.
When $N=\frac{v^k}{N}$,  we denote an  DLOA $(\frac{v^k}{N},N; t, k,v)$ by
 DLOA$(N; t, k, v)$ in short.
  Further, a  DLOA $(N; t, k, v)$, $\{A_0$, $A_1$, $\cdots$, $A_{N-1}\}$ is called {\it strong}, denoted by SDLOA $(N; t, k, v)$, if $D, D'$ are both OA$(N; t, k, v)$s, where $D$ and $D'$ are two arrays whose $j$-th columns are the same as $j$-th and $(N-1-j)$-th
 column of $A_j$, respectively, i.e.,
  \begin{center}
$\begin{aligned}
\mbox{}\hspace{1.55in}
&D=(d_{i,j})_{k\times N},  d_{i,j}=a_{i,j}^{(j)},  i\in I_k, j\in I_N, \mbox{}\hspace{1.75in}  \\
&D'=(d'_{i,j})_{k\times N}, d'_{i,j}=a_{i,N-1-j}^{(j)}, i\in I_k, j\in I_N. \mbox{}\hspace{1.37in}
\end{aligned}$
\end{center}

 We should mention that DLOA is a powerful tool to construct multimagic rectangles \cite{ZhangMR}. However, this paper mainly focus on multimagic squares. Zhang, Chen and Lei  investigated the  construction of multimagic squares based on SDLOA, and they proved the following.
\begin{lem} \emph{(\cite{Zhang})}\label{zhang}
There exists an MS$(q^{t},t)$ for all prime power $q \geq 2t-1$ with $t\geq2$.
 \end{lem}

In this paper, a further result of  multimagic squares based on SDLOA  is obtained as follows.

\begin{thm}\label{thm-q2t-1}\label{q2t-1}
There exists an MS$(q^{2t-1},t)$ for any prime power $q\geq 2t-1$ with $t\geq3$.
\end{thm}

 Some useful constructions of multimagic squares are reviewed in Section 2. A family of complementary $t$-multimagic squares are  given in Section 3.
  The proofs of  Theorem \ref{q2t-1} and some useful corollaries are provided in Section 4.

\section{Constructions of multimagic squares}

In this section, we review some useful construction of multimagic squares presented in  \cite{Derksen,Zhang,Zhangcms}.
The construction of  multimagic squares based on SDLOA is listed in the following.
\begin{con}\label{main2} \emph{\cite{Zhang}}
If there exists an SDLOA $(N; t, k, v)$, then there exists an MS$(N,t)$.
\end{con}
 Derksen et al \cite{Derksen}  proved the following product construction of multimagic squares.
  \begin{con} \emph{(\cite{Derksen})}\label{product}  \label{conproduct}
If there exists  an  MS$(m,t)$ and there exists  an  MS($n,t$), then there exists an MS$(mn,t)$.
\end{con}

 Complementary $t$-multimagic squares was introduced in \cite{Zhangcms}. Let $I_n=\{0,1,\cdots, n-1\}$. For an MS$(n,t)$ $A$, the magic number of $A^{*e}$ is denoted by $S_e(n)$, where $$S_e(n)=\frac{\sum_{k\in I_{n^2}}k^e}{n}, e=1,2,\cdots,t.$$

 Let $B_s=(b_{i,j}^{(s)})$, $s\in I_m$, and let $B_0, B_1, \cdots, B_{m-1}$ be $m$ MS($n,t$)s. Then $\{B_0$, $B_1$, $\cdots$, $B_{m-1}\}$ is called {\it a set of complementary $t$-multimagic squares}, denoted by $m$-CMS($n,t$), if the following three conditions are satisfied.
\vskip 5pt

 \noindent \qquad $(R1)$ \ \ \ \ \ \ \ \ \ \ \ \ \ \  \   $\sum\limits_{s\in I_m}\sum\limits_{j\in I_n}(b_{i,j}^{(s)})^{t+1}=mS_{t+1}(n),   i\in I_n$; \vskip 5pt

 \noindent \qquad $(R2)$ \ \ \ \ \ \ \ \ \ \ \ \ \ \  \   $\sum\limits_{s\in I_m}\sum\limits_{i\in I_n}(b_{i,j}^{(s)})^{t+1}=mS_{t+1}(n),   j\in I_n$; \vskip 5pt

 \noindent \qquad $(R3)$ \ \ \ \ \ \ \ \ \ \ \ \ \ \  \  $\sum\limits_{s\in I_m}\sum\limits_{i\in I_n}(b_{i,i}^{(s)})^{t+1}=\sum\limits_{s\in I_m}\sum\limits_{i\in I_n}(b_{i,n-1-i}^{(s)})^{t+1}=mS_{t+1}(n)$.

  \vskip 10pt
  A recursive construction of $t$-multimagic squares were provided in \cite{Zhangcms} by using CMS.
\begin{con} \emph{\cite{Zhangcms}} \label{conmscms}
  If there is an  MS$(m,t)$ and there is an $m$-CMS$(n,t-1)$,  then there is an MS$(mn,t)$.
\end{con}

\section{Complementary $t$-multimagic squares based on strong double large sets of orthogonal arrays}

In this section, we   give a family of CMS with prime power order   by making use   of SDLOA.

 Let $F_q$ be a Galois field with $q$ elements.
  Let $M_{k\times s}^{(t)}(F_q)$ be the set of $k\times s$ matrices over  $F_q$ with the property that any $t$ rows are linearly independent.
The following results are useful.

\begin{lem}\emph{\label{LOA1}\cite{Zhang}}
  If there exists a  nonsingular  matrix $E=(e_{i,j})_{k\times k}$ over $F_q$ with a submatrix $E_1\in M_{k\times s}^{(t)}(F_q)$,  then there exists an LOA$(q^s; t, k, q)$.
 \end{lem}

\begin{lem} \label{OA1}\emph{\cite{Zhang}}
 Let   $k=2s$, $2\leq t\leq s$. If  there exists a   nonsingular
 matrix $E=(e_{i,j})_{k\times k}=(E_1, E_2)$ over $F_q$ such that  $E_1$, $E_2$, $E_1+E_2$, $E_1-E_2$   are  all in  $M_{k\times s}^{(t)}(F_q)$,  then  there exists an  SDLOA $(q^s; t, k, q)$.
 \end{lem}

Let $F_q^k$ be the set of all $k$-dimensional column vectors over $F_q$.
Denote the set of a complete set of congruence classes modulo $n$ by $\mathbb{Z}_n$, i.e.,  $\mathbb{Z}_n=\{\overline{0},\overline{1},\cdots,\overline{n-1}\}$.   It is known that $\mathbb{Z}_n$ is a field whenever $n$ is prime.

 Let $q=p^m$, $p$ is prime, $ n\geq1$. A Galois field $F_q$ can be written as a vector space over $\mathbb{Z}_p$, i.e.,
$F_q=\{\xi_j|\xi_j=(\overline w_{0},\overline w_{1},\cdots,\overline w_{n-1})^T\in \mathbb{Z}_p^m, j=\sum\limits_{i=0}^{m-1}w_ip^i\}.$
We have $\xi_j+\xi_{q-1-j}=\xi_{q-1}$, $j\in I_q$.
Denote $(\xi_{q-1},\xi_{q-1},\cdots,\xi_{q-1})^T\in F_q^t$ by $\widetilde{X}$.

\vskip 5pt
We shall use SDLOA$(q^t; t, 2t, q)$ to construct  $q^t$-CMS$(q^t,t)$.
For convenience,  an  SDLOA $(q^t; t, 2t, q)$ will be written  as a $2t\times q^t \times q^t$ array. Further, a $2t\times q^t \times q^t$ array can be viewed as  a $q^t \times q^t$ matrix $C$, where each element of $C$ is a $2t$-dimensional  vector.
Clearly, such a $q^t \times q^t$ matrix  $C$ is an SDLOA$(q^t; t, 2t, q)$ if it has the property that each row, each column, main diagonal and back diagonal of $C$ forms an OA$(q^t; t, 2t, q)$,
and all the rows of $C$ form an LOA$(q^t; t, 2t, q)$.
We shall index the rows and the columns of $C$ by the vectors $(x,y)^T\in F_q^t$ or just by $x+qy$.
We now give a $9$-CMS$(9,2)$ by using SDLOA.
    \begin{lem}\label{cms9}
There exists a  $9$-CMS$(9,2)$.
    \end{lem}
   \begin{proof}
Let  {$E_1=\left( \begin{smallmatrix}
1&0\\
0&1\\
1&1\\
2&1\\
 \end{smallmatrix}\right)$,
 $E_2=\left( \begin{smallmatrix}
0&1\\
2&0\\
1&2\\
1&1\\
 \end{smallmatrix}\right)$ and $E=(E_1, E_2)$. We can verify that $E$ is nonsingular and $E_1, E_2, E_1+E_2, E_1-E_2\in M_{4\times 2}^{(2)}(\mathbb{Z}_3)$.}
We take
  {\renewcommand\arraystretch{0.8}
\setlength{\arraycolsep}{1.5pt}
 \small
\begin{center}
$H_{0}=\binom{0}{0},H_{1}=\binom{0}{1},H_{2}=\binom{0}{2},
 H_{3}=\binom{1}{0},H_{4}=\binom{1}{1},H_{5}=\binom{1}{2},
 H_{6}=\binom{2}{0},H_{7}=\binom{2}{1},H_{8}=\binom{2}{2};$
\end{center}}
  {\renewcommand\arraystretch{0.8}
\setlength{\arraycolsep}{1.5pt}
 \small
\begin{center}
$H^*_{0}=\binom{0}{1},H^*_{1}=\binom{0}{2},H^*_{2}=\binom{0}{0},
 H^*_{3}=\binom{2}{1},H^*_{4}=\binom{2}{2},H^*_{5}=\binom{2}{0},
 H^*_{6}=\binom{1}{1},H^*_{7}=\binom{1}{2},H^*_{8}=\binom{1}{0}.$
\end{center}}
 \noindent Let
 \begin{center}
$C_{i}=(C^{(i)}_{X,Y})$,
$C^{(i)}_{X,Y}=(E_1, E_2)
\left(
\begin{array}{l}
X+H_i\\
Y+H_{i}^{*}
\end{array}
\right)
=E_1(X+H_{i})+E_2(Y+H_{i}^{*}), X,Y\in \mathbb{Z}_3^2.$
\end{center}
We list the $C_0,C_1,\cdots,C_8$ in the following.  For convenience, each entry of $C_i$ is represented as a row vector.
 {{\renewcommand\arraystretch{0.8}
\setlength{\arraycolsep}{0.7pt}
\footnotesize
\begin{center}
$C_0=\left(
        \begin{array}{c|c|c|c|c|c|c|c|c }
1021&2012&0000&1202&2220&0211&1110&2101&0122\\
1102&2120&0111&1010&2001&0022&1221&2212&0200\\
1210&2201&0222&1121&2112&0100&1002&2020&0011\\
2000&0021&1012&2211&0202&1220&2122&0110&1101\\
2111&0102&1120&2022&0010&1001&2200&0221&1212\\
2222&0210&1201&2100&0121&1112&2011&0002&1020\\
0012&1000&2021&0220&1211&2202&0101&1122&2110\\
0120&1111&2102&0001&1022&2010&0212&1200&2221\\
0201&1222&2210&0112&1100&2121&0020&1011&2002
        \end{array}
        \right)$,
        $C_1=\left(
        \begin{array}{c|c|c|c|c|c|c|c|c c}
2120&0111&1102&2001&0022&1010&2212&0200&1221\\
2201&0222&1210&2112&0100&1121&2020&0011&1002\\
2012&0000&1021&2220&0211&1202&2101&0122&1110&\\
0102&1120&2111&0010&1001&2022&0221&1212&2200&\\
0210&1201&2222&0121&1112&2100&0002&1020&2011&\\
0021&1012&2000&0202&1220&2211&0110&1101&2122&\\
1111&2102&0120&1022&2010&0001&1200&2221&0212&\\
1222&2210&0201&1100&2121&0112&1011&2002&0020&\\
1000&2021&0012&1211&2202&0220&1122&2110&0101&
        \end{array}
        \right)$,
                $C_2=\left(
        \begin{array}{c|c|c|c|c|c|c|c|c c}
0222&1210&2201&0100&1121&2112&0011&1002&2020&\\
0000&1021&2012&0211&1202&2220&0122&1110&2101&\\
0111&1102&2120&0022&1010&2001&0200&1221&2212&\\
1201&2222&0210&1112&2100&0121&1020&2011&0002&\\
1012&2000&0021&1220&2211&0202&1101&2122&0110&\\
1120&2111&0102&1001&2022&0010&1212&2200&0221&\\
2210&0201&1222&2121&0112&1100&2002&0020&1011&\\
2021&0012&1000&2202&0220&1211&2110&0101&1122&\\
2102&0120&1111&2010&0001&1022&2221&0212&1200&
        \end{array}
        \right)$,
                $C_3=\left(
        \begin{array}{c|c|c|c|c|c|c|c|cc }
2122&0110&1101&2000&0021&1012&2211&0202&1220&\\
2200&0221&1212&2111&0102&1120&2022&0010&1001&\\
2011&0002&1020&2222&0210&1201&2100&0121&1112&\\
0101&1122&2110&0012&1000&2021&0220&1211&2202&\\
0212&1200&2221&0120&1111&2102&0001&1022&2010&\\
0020&1011&2002&0201&1222&2210&0112&1100&2121&\\
1110&2101&0122&1021&2012&0000&1202&2220&0211&\\
1221&2212&0200&1102&2120&0111&1010&2001&0022&\\
1002&2020&0011&1210&2201&0222&1121&2112&0100&
        \end{array}
        \right)$,
                $C_4=\left(
        \begin{array}{c|c|c|c|c|c|c|c|cc }
0221&1212&2200&0102&1120&2111&0010&1001&2022&\\
0002&1020&2011&0210&1201&2222&0121&1112&2100&\\
0110&1101&2122&0021&1012&2000&0202&1220&2211&\\
1200&2221&0212&1111&2102&0120&1022&2010&0001&\\
1011&2002&0020&1222&2210&0201&1100&2121&0112&\\
1122&2110&0101&1000&2021&0012&1211&2202&0220&\\
2212&0200&1221&2120&0111&1102&2001&0022&1010&\\
2020&0011&1002&2201&0222&1210&2112&0100&1121&\\
2101&0122&1110&2012&0000&1021&2220&0211&1202&
        \end{array}
        \right)$,
                $C_5=\left(
        \begin{array}{c|c|c|c|c|c|c|c|cc }
1020&2011&0002&1201&2222&0210&1112&2100&0121&\\
1101&2122&0110&1012&2000&0021&1220&2211&0202&\\
1212&2200&0221&1120&2111&0102&1001&2022&0010&\\
2002&0020&1011&2210&0201&1222&2121&0112&1100&\\
2110&0101&1122&2021&0012&1000&2202&0220&1211&\\
2221&0212&1200&2102&0120&1111&2010&0001&1022&\\
0011&1002&2020&0222&1210&2201&0100&1121&2112&\\
0122&1110&2101&0000&1021&2012&0211&1202&2220&\\
0200&1221&2212&0111&1102&2120&0022&1010&2001&
        \end{array}
        \right)$,
                $C_6=\left(
        \begin{array}{c|c|c|c|c|c|c|c|c c}
0220&1211&2202&0101&1122&2110&0012&1000&2021&\\
0001&1022&2010&0212&1200&2221&0120&1111&2102&\\
0112&1100&2121&0020&1011&2002&0201&1222&2210&\\
1202&2220&0211&1110&2101&0122&1021&2012&0000&\\
1010&2001&0022&1221&2212&0200&1102&2120&0111&\\
1121&2112&0100&1002&2020&0011&1210&2201&0222&\\
2211&0202&1220&2122&0110&1101&2000&0021&1012&\\
2022&0010&1001&2200&0221&1212&2111&0102&1120&\\
2100&0121&1112&2011&0002&1020&2222&0210&1201&
        \end{array}
        \right)$,
                        $C_7=\left(
        \begin{array}{c|c|c|c|c|c|c|c|cc }
1022&2010&0001&1200&2221&0212&1111&2102&0120&\\
1100&2121&0112&1011&2002&0020&1222&2210&0201&\\
1211&2202&0220&1122&2110&0101&1000&2021&0012&\\
2001&0022&1010&2212&0200&1221&2120&0111&1102&\\
2112&0100&1121&2020&0011&1002&2201&0222&1210&\\
2220&0211&1202&2101&0122&1110&2012&0000&1021&\\
0010&1001&2022&0221&1212&2200&0102&1120&2111&\\
0121&1112&2100&0002&1020&2011&0210&1201&2222&\\
0202&1220&2211&0110&1101&2122&0021&1012&2000&
        \end{array}
        \right)$,
                        $C_8=\left(
        \begin{array}{c|c|c|c|c|c|c|c|c c}
2121&0112&1100&2002&0020&1011&2210&0201&1222&\\
2202&0220&1211&2110&0101&1122&2021&0012&1000&\\
2010&0001&1022&2221&0212&1200&2102&0120&1111&\\
0100&1121&2112&0011&1002&2020&0222&1210&2201&\\
0211&1202&2220&0122&1110&2101&0000&1021&2012&\\
0022&1010&2001&0200&1221&2212&0111&1102&2120&\\
1112&2100&0121&1020&2011&0002&1201&2222&0210&\\
1220&2211&0202&1101&2122&0110&1012&2000&0021&\\
1001&2022&0010&1212&2200&0221&1120&2111&0102&
        \end{array}
        \right)\ \ \ \ \ \ \ \  \ \ \  \ \ \ \ \ \ \ \ \ \ \  \ \ \  \ \ \ \ \ \ \ \ \ \ \  \ \ \  \ \ \ \ \ \ \ \ \ \ \  \ \ \  \ \ \ \ \ \ \ \ \ \ \ \ \ \  \ \ \  \ \ \  $.
\end{center}}}
\noindent For any $i,k,l\in I_9$,  if we write the $(k,l)$-entry of $C_i$  in a column vector form, then each row of $C_i$ forms  a $4\times 9$ matrix. Thus we get   81  $4\times 9$ matrix $A_{ik}$ with the property that the $l$-th column of $A_{ik}$ is the transpose of the vector of the $(k,l)$-entry of $C_i$.
 We list all the $4\times 9$ matrices as follows.
{{\renewcommand\arraystretch{0.6}
\setlength{\arraycolsep}{1.0pt}
\footnotesize
\begin{center}
$A_{00}=\left(
        \begin{array}{ccccccccc }
1&2&0&1&2&0&1&2&0\\0&0&0&2&2&2&1&1&1\\2&1&0&0&2&1&1&0&2\\1&2&0&2&0&1&0&1&2
        \end{array}
        \right)$,
        $A_{01}=\left(
        \begin{array}{ccccccccc }
1&2&0&1&2&0&1&2&0\\1&1&1&0&0&0&2&2&2\\0&2&1&1&0&2&2&1&0\\2&0&1&0&1&2&1&2&0
        \end{array}
        \right)$,
                $A_{02}=\left(
        \begin{array}{ccccccccc }
1&2&0&1&2&0&1&2&0\\2&2&2&1&1&1&0&0&0\\1&0&2&2&1&0&0&2&1\\0&1&2&1&2&0&2&0&1
        \end{array}
        \right)$,\\
                $A_{03}=\left(
        \begin{array}{ccccccccc }
2&0&1&2&0&1&2&0&1\\0&0&0&2&2&2&1&1&1\\0&2&1&1&0&2&2&1&0\\0&1&2&1&2&0&2&0&1
        \end{array}
        \right)$,
                $A_{04}=\left(
        \begin{array}{ccccccccc }
2&0&1&2&0&1&2&0&1\\1&1&1&0&0&0&2&2&2\\1&0&2&2&1&0&0&2&1\\1&2&0&2&0&1&0&1&2
        \end{array}
        \right)$,
                $A_{05}=\left(
        \begin{array}{ccccccccc }
2&0&1&2&0&1&2&0&1\\2&2&2&1&1&1&0&0&0\\2&1&0&0&2&1&1&0&2\\2&0&1&0&1&2&1&2&0
        \end{array}
        \right)$,\\
                $A_{06}=\left(
        \begin{array}{ccccccccc }
0&1&2&0&1&2&0&1&2\\0&0&0&2&2&2&1&1&1\\1&0&2&2&1&0&0&2&1\\2&0&1&0&1&2&1&2&0
        \end{array}
        \right)$,
                $A_{07}=\left(
        \begin{array}{ccccccccc }
0&1&2&0&1&2&0&1&2\\1&1&1&0&0&0&2&2&2\\2&1&0&0&2&1&1&0&2\\0&1&2&1&2&0&2&0&1
        \end{array}
        \right)$,
                $A_{08}=\left(
        \begin{array}{ccccccccc }
0&1&2&0&1&2&0&1&2\\2&2&2&1&1&1&0&0&0\\0&2&1&1&0&2&2&1&0\\1&2&0&2&0&1&0&1&2
        \end{array}
        \right).$\\
        $A_{10}=\left(
        \begin{array}{ccccccccc }
2&0&1&2&0&1&2&0&1\\1&1&1&0&0&0&2&2&2\\2&1&0&0&2&1&1&0&2\\0&1&2&1&2&0&2&0&1\\
        \end{array}
        \right)$,
        $A_{11}=\left(
        \begin{array}{ccccccccc }
2&0&1&2&0&1&2&0&1\\2&2&2&1&1&1&0&0&0\\0&2&1&1&0&2&2&1&0\\1&2&0&2&0&1&0&1&2\\
        \end{array}
        \right)$,
                $A_{12}=\left(
        \begin{array}{ccccccccc }
2&0&1&2&0&1&2&0&1\\0&0&0&2&2&2&1&1&1\\1&0&2&2&1&0&0&2&1\\2&0&1&0&1&2&1&2&0\\
        \end{array}
        \right)$,\\
                $A_{13}=\left(
        \begin{array}{ccccccccc }
0&1&2&0&1&2&0&1&2\\1&1&1&0&0&0&2&2&2\\0&2&1&1&0&2&2&1&0\\2&0&1&0&1&2&1&2&0\\
        \end{array}
        \right)$,
                $A_{14}=\left(
        \begin{array}{ccccccccc }
0&1&2&0&1&2&0&1&2\\2&2&2&1&1&1&0&0&0\\1&0&2&2&1&0&0&2&1\\0&1&2&1&2&0&2&0&1\\
        \end{array}
        \right)$,
                $A_{15}=\left(
        \begin{array}{ccccccccc }
0&1&2&0&1&2&0&1&2\\0&0&0&2&2&2&1&1&1\\2&1&0&0&2&1&1&0&2\\1&2&0&2&0&1&0&1&2\\
        \end{array}
        \right)$,\\
                $A_{16}=\left(
        \begin{array}{ccccccccc }
1&2&0&1&2&0&1&2&0\\1&1&1&0&0&0&2&2&2\\1&0&2&2&1&0&0&2&1\\1&2&0&2&0&1&0&1&2\\
        \end{array}
        \right)$,
                $A_{17}=\left(
        \begin{array}{ccccccccc }
1&2&0&1&2&0&1&2&0\\2&2&2&1&1&1&0&0&0\\2&1&0&0&2&1&1&0&2\\2&0&1&0&1&2&1&2&0\\
        \end{array}
        \right)$,
                $A_{18}=\left(
        \begin{array}{ccccccccc }
1&2&0&1&2&0&1&2&0\\0&0&0&2&2&2&1&1&1\\0&2&1&1&0&2&2&1&0\\0&1&2&1&2&0&2&0&1\\
        \end{array}
        \right). $\\
                $A_{20}=\left(
        \begin{array}{ccccccccc }
0&1&2&0&1&2&0&1&2\\2&2&2&1&1&1&0&0&0\\2&1&0&0&2&1&1&0&2\\2&0&1&0&1&2&1&2&0\\
        \end{array}
        \right)$,
        $A_{21}=\left(
        \begin{array}{ccccccccc }
0&1&2&0&1&2&0&1&2\\0&0&0&2&2&2&1&1&1\\0&2&1&1&0&2&2&1&0\\0&1&2&1&2&0&2&0&1\\
        \end{array}
        \right)$,
                $A_{22}=\left(
        \begin{array}{ccccccccc }
0&1&2&0&1&2&0&1&2\\1&1&1&0&0&0&2&2&2\\1&0&2&2&1&0&0&2&1\\1&2&0&2&0&1&0&1&2\\
        \end{array}
        \right)$,\\
                $A_{23}=\left(
        \begin{array}{ccccccccc }
1&2&0&1&2&0&1&2&0\\2&2&2&1&1&1&0&0&0\\0&2&1&1&0&2&2&1&0\\1&2&0&2&0&1&0&1&2\\
        \end{array}
        \right)$,
                $A_{24}=\left(
        \begin{array}{ccccccccc }
1&2&0&1&2&0&1&2&0\\0&0&0&2&2&2&1&1&1\\1&0&2&2&1&0&0&2&1\\2&0&1&0&1&2&1&2&0\\
        \end{array}
        \right)$,
                $A_{25}=\left(
        \begin{array}{ccccccccc }
1&2&0&1&2&0&1&2&0\\1&1&1&0&0&0&2&2&2\\2&1&0&0&2&1&1&0&2\\0&1&2&1&2&0&2&0&1\\
        \end{array}
        \right)$,\\
                $A_{26}=\left(
        \begin{array}{ccccccccc }
2&0&1&2&0&1&2&0&1\\2&2&2&1&1&1&0&0&0\\1&0&2&2&1&0&0&2&1\\0&1&2&1&2&0&2&0&1\\
        \end{array}
        \right)$,
                $A_{27}=\left(
        \begin{array}{ccccccccc }
2&0&1&2&0&1&2&0&1\\0&0&0&2&2&2&1&1&1\\2&1&0&0&2&1&1&0&2\\1&2&0&2&0&1&0&1&2\\
        \end{array}
        \right)$,
                $A_{28}=\left(
        \begin{array}{ccccccccc }
2&0&1&2&0&1&2&0&1\\1&1&1&0&0&0&2&2&2\\0&2&1&1&0&2&2&1&0\\2&0&1&0&1&2&1&2&0\\
        \end{array}
        \right).$\\
                $A_{30}=\left(
        \begin{array}{ccccccccc }
2&0&1&2&0&1&2&0&1\\1&1&1&0&0&0&2&2&2\\2&1&0&0&2&1&1&0&2\\2&0&1&0&1&2&1&2&0\\
        \end{array}
        \right)$,
        $A_{31}=\left(
        \begin{array}{ccccccccc }
2&0&1&2&0&1&2&0&1\\2&2&2&1&1&1&0&0&0\\0&2&1&1&0&2&2&1&0\\0&1&2&1&2&0&2&0&1\\
        \end{array}
        \right)$,
                $A_{32}=\left(
        \begin{array}{ccccccccc }
2&0&1&2&0&1&2&0&1\\0&0&0&2&2&2&1&1&1\\1&0&2&2&1&0&0&2&1\\1&2&0&2&0&1&0&1&2\\
        \end{array}
        \right)$,\\
                $A_{33}=\left(
        \begin{array}{ccccccccc }
0&1&2&0&1&2&0&1&2\\1&1&1&0&0&0&2&2&2\\0&2&1&1&0&2&2&1&0\\1&2&0&2&0&1&0&1&2\\
        \end{array}
        \right)$,
                $A_{34}=\left(
        \begin{array}{ccccccccc }
0&1&2&0&1&2&0&1&2\\2&2&2&1&1&1&0&0&0\\1&0&2&2&1&0&0&2&1\\2&0&1&0&1&2&1&2&0\\
        \end{array}
        \right)$,
                $A_{35}=\left(
        \begin{array}{ccccccccc }
0&1&2&0&1&2&0&1&2\\0&0&0&2&2&2&1&1&1\\2&1&0&0&2&1&1&0&2\\0&1&2&1&2&0&2&0&1\\
        \end{array}
        \right)$,\\
                $A_{36}=\left(
        \begin{array}{ccccccccc }
1&2&0&1&2&0&1&2&0\\1&1&1&0&0&0&2&2&2\\1&0&2&2&1&0&0&2&1\\0&1&2&1&2&0&2&0&1\\
        \end{array}
        \right)$,
                $A_{37}=\left(
        \begin{array}{ccccccccc }
1&2&0&1&2&0&1&2&0\\2&2&2&1&1&1&0&0&0\\2&1&0&0&2&1&1&0&2\\1&2&0&2&0&1&0&1&2\\
        \end{array}
        \right)$,
                $A_{38}=\left(
        \begin{array}{ccccccccc }
1&2&0&1&2&0&1&2&0\\0&0&0&2&2&2&1&1&1\\0&2&1&1&0&2&2&1&0\\2&0&1&0&1&2&1&2&0\\
        \end{array}
        \right). $\\
                $A_{40}=\left(
        \begin{array}{ccccccccc }
0&1&2&0&1&2&0&1&2\\2&2&2&1&1&1&0&0&0\\2&1&0&0&2&1&1&0&2\\1&2&0&2&0&1&0&1&2\\
        \end{array}
        \right)$,
        $A_{41}=\left(
        \begin{array}{ccccccccc }
0&1&2&0&1&2&0&1&2\\0&0&0&2&2&2&1&1&1\\0&2&1&1&0&2&2&1&0\\2&0&1&0&1&2&1&2&0\\
        \end{array}
        \right)$,
                $A_{42}=\left(
        \begin{array}{ccccccccc }
0&1&2&0&1&2&0&1&2\\1&1&1&0&0&0&2&2&2\\1&0&2&2&1&0&0&2&1\\0&1&2&1&2&0&2&0&1\\
        \end{array}
        \right)$,\\
                $A_{43}=\left(
        \begin{array}{ccccccccc }
1&2&0&1&2&0&1&2&0\\2&2&2&1&1&1&0&0&0\\0&2&1&1&0&2&2&1&0\\0&1&2&1&2&0&2&0&1\\
        \end{array}
        \right)$,
                $A_{44}=\left(
        \begin{array}{ccccccccc }
1&2&0&1&2&0&1&2&0\\0&0&0&2&2&2&1&1&1\\1&0&2&2&1&0&0&2&1\\1&2&0&2&0&1&0&1&2\\
        \end{array}
        \right)$,
                $A_{45}=\left(
        \begin{array}{ccccccccc }
1&2&0&1&2&0&1&2&0\\1&1&1&0&0&0&2&2&2\\2&1&0&0&2&1&1&0&2\\2&0&1&0&1&2&1&2&0\\
        \end{array}
        \right)$,\\
                $A_{46}=\left(
        \begin{array}{ccccccccc }
2&0&1&2&0&1&2&0&1\\2&2&2&1&1&1&0&0&0\\1&0&2&2&1&0&0&2&1\\2&0&1&0&1&2&1&2&0\\
        \end{array}
        \right)$,
                $A_{47}=\left(
        \begin{array}{ccccccccc }
2&0&1&2&0&1&2&0&1\\0&0&0&2&2&2&1&1&1\\2&1&0&0&2&1&1&0&2\\0&1&2&1&2&0&2&0&1\\
        \end{array}
        \right)$,
                $A_{48}=\left(
        \begin{array}{ccccccccc }
2&0&1&2&0&1&2&0&1\\1&1&1&0&0&0&2&2&2\\0&2&1&1&0&2&2&1&0\\1&2&0&2&0&1&0&1&2\\
        \end{array}
        \right). $\\
                $A_{50}=\left(
        \begin{array}{ccccccccc }
1&2&0&1&2&0&1&2&0\\0&0&0&2&2&2&1&1&1\\2&1&0&0&2&1&1&0&2\\0&1&2&1&2&0&2&0&1\\
        \end{array}
        \right)$,
        $A_{51}=\left(
        \begin{array}{ccccccccc }
1&2&0&1&2&0&1&2&0\\1&1&1&0&0&0&2&2&2\\0&2&1&1&0&2&2&1&0\\1&2&0&2&0&1&0&1&2\\
        \end{array}
        \right)$,
                $A_{52}=\left(
        \begin{array}{ccccccccc }
1&2&0&1&2&0&1&2&0\\2&2&2&1&1&1&0&0&0\\1&0&2&2&1&0&0&2&1\\2&0&1&0&1&2&1&2&0\\
        \end{array}
        \right)$,\\
                $A_{53}=\left(
        \begin{array}{ccccccccc }
2&0&1&2&0&1&2&0&1\\0&0&0&2&2&2&1&1&1\\0&2&1&1&0&2&2&1&0\\2&0&1&0&1&2&1&2&0\\
        \end{array}
        \right)$,
                $A_{54}=\left(
        \begin{array}{ccccccccc }
2&0&1&2&0&1&2&0&1\\1&1&1&0&0&0&2&2&2\\1&0&2&2&1&0&0&2&1\\0&1&2&1&2&0&2&0&1\\
        \end{array}
        \right)$,
                $A_{55}=\left(
        \begin{array}{ccccccccc }
2&0&1&2&0&1&2&0&1\\2&2&2&1&1&1&0&0&0\\2&1&0&0&2&1&1&0&2\\1&2&0&2&0&1&0&1&2\\
        \end{array}
        \right)$,\\
                $A_{56}=\left(
        \begin{array}{ccccccccc }
0&1&2&0&1&2&0&1&2\\0&0&0&2&2&2&1&1&1\\1&0&2&2&1&0&0&2&1\\1&2&0&2&0&1&0&1&2\\
        \end{array}
        \right)$,
                $A_{57}=\left(
        \begin{array}{ccccccccc }
0&1&2&0&1&2&0&1&2\\1&1&1&0&0&0&2&2&2\\2&1&0&0&2&1&1&0&2\\2&0&1&0&1&2&1&2&0\\
        \end{array}
        \right)$,
                $A_{58}=\left(
        \begin{array}{ccccccccc }
0&1&2&0&1&2&0&1&2\\2&2&2&1&1&1&0&0&0\\0&2&1&1&0&2&2&1&0\\0&1&2&1&2&0&2&0&1\\
        \end{array}
        \right). $\\
                $A_{60}=\left(
        \begin{array}{ccccccccc }
0&1&2&0&1&2&0&1&2\\2&2&2&1&1&1&0&0&0\\2&1&0&0&2&1&1&0&2\\0&1&2&1&2&0&2&0&1\\
        \end{array}
        \right)$,
        $A_{61}=\left(
        \begin{array}{ccccccccc }
0&1&2&0&1&2&0&1&2\\0&0&0&2&2&2&1&1&1\\0&2&1&1&0&2&2&1&0\\1&2&0&2&0&1&0&1&2\\
        \end{array}
        \right)$,
                $A_{62}=\left(
        \begin{array}{ccccccccc }
0&1&2&0&1&2&0&1&2\\1&1&1&0&0&0&2&2&2\\1&0&2&2&1&0&0&2&1\\2&0&1&0&1&2&1&2&0\\
        \end{array}
        \right)$,\\
                $A_{63}=\left(
        \begin{array}{ccccccccc }
1&2&0&1&2&0&1&2&0\\2&2&2&1&1&1&0&0&0\\0&2&1&1&0&2&2&1&0\\2&0&1&0&1&2&1&2&0\\
        \end{array}
        \right)$,
                $A_{64}=\left(
        \begin{array}{ccccccccc }
1&2&0&1&2&0&1&2&0\\0&0&0&2&2&2&1&1&1\\1&0&2&2&1&0&0&2&1\\0&1&2&1&2&0&2&0&1\\
        \end{array}
        \right)$,
                $A_{65}=\left(
        \begin{array}{ccccccccc }
1&2&0&1&2&0&1&2&0\\1&1&1&0&0&0&2&2&2\\2&1&0&0&2&1&1&0&2\\1&2&0&2&0&1&0&1&2\\
        \end{array}
        \right)$,\\
                $A_{66}=\left(
        \begin{array}{ccccccccc }
2&0&1&2&0&1&2&0&1\\2&2&2&1&1&1&0&0&0\\1&0&2&2&1&0&0&2&1\\1&2&0&2&0&1&0&1&2\\
        \end{array}
        \right)$,
                $A_{67}=\left(
        \begin{array}{ccccccccc }
2&0&1&2&0&1&2&0&1\\0&0&0&2&2&2&1&1&1\\2&1&0&0&2&1&1&0&2\\2&0&1&0&1&2&1&2&0\\
        \end{array}
        \right)$,
                $A_{68}=\left(
        \begin{array}{ccccccccc }
2&0&1&2&0&1&2&0&1\\1&1&1&0&0&0&2&2&2\\0&2&1&1&0&2&2&1&0\\0&1&2&1&2&0&2&0&1\\
        \end{array}
        \right). $\\
                $A_{70}=\left(
        \begin{array}{ccccccccc }
1&2&0&1&2&0&1&2&0\\0&0&0&2&2&2&1&1&1\\2&1&0&0&2&1&1&0&2\\2&0&1&0&1&2&1&2&0\\
        \end{array}
        \right)$,
        $A_{71}=\left(
        \begin{array}{ccccccccc }
1&2&0&1&2&0&1&2&0\\1&1&1&0&0&0&2&2&2\\0&2&1&1&0&2&2&1&0\\0&1&2&1&2&0&2&0&1\\
        \end{array}
        \right)$,
                $A_{72}=\left(
        \begin{array}{ccccccccc }
1&2&0&1&2&0&1&2&0\\2&2&2&1&1&1&0&0&0\\1&0&2&2&1&0&0&2&1\\1&2&0&2&0&1&0&1&2\\
        \end{array}
        \right)$,\\
                $A_{73}=\left(
        \begin{array}{ccccccccc }
2&0&1&2&0&1&2&0&1\\0&0&0&2&2&2&1&1&1\\0&2&1&1&0&2&2&1&0\\1&2&0&2&0&1&0&1&2\\
        \end{array}
        \right)$,
                $A_{74}=\left(
        \begin{array}{ccccccccc }
2&0&1&2&0&1&2&0&1\\1&1&1&0&0&0&2&2&2\\1&0&2&2&1&0&0&2&1\\2&0&1&0&1&2&1&2&0\\
        \end{array}
        \right)$,
                $A_{75}=\left(
        \begin{array}{ccccccccc }
2&0&1&2&0&1&2&0&1\\2&2&2&1&1&1&0&0&0\\2&1&0&0&2&1&1&0&2\\0&1&2&1&2&0&2&0&1\\
        \end{array}
        \right)$,\\
                $A_{76}=\left(
        \begin{array}{ccccccccc }
0&1&2&0&1&2&0&1&2\\0&0&0&2&2&2&1&1&1\\1&0&2&2&1&0&0&2&1\\0&1&2&1&2&0&2&0&1\\
        \end{array}
        \right)$,
                $A_{77}=\left(
        \begin{array}{ccccccccc }
0&1&2&0&1&2&0&1&2\\1&1&1&0&0&0&2&2&2\\2&1&0&0&2&1&1&0&2\\1&2&0&2&0&1&0&1&2\\
        \end{array}
        \right)$,
                $A_{78}=\left(
        \begin{array}{ccccccccc }
0&1&2&0&1&2&0&1&2\\2&2&2&1&1&1&0&0&0\\0&2&1&1&0&2&2&1&0\\2&0&1&0&1&2&1&2&0\\
        \end{array}
        \right). $\\
                $A_{80}=\left(
        \begin{array}{ccccccccc }
2&0&1&2&0&1&2&0&1\\1&1&1&0&0&0&2&2&2\\2&1&0&0&2&1&1&0&2\\1&2&0&2&0&1&0&1&2\\
        \end{array}
        \right)$,
        $A_{81}=\left(
        \begin{array}{ccccccccc }
2&0&1&2&0&1&2&0&1\\2&2&2&1&1&1&0&0&0\\0&2&1&1&0&2&2&1&0\\2&0&1&0&1&2&1&2&0\\
        \end{array}
        \right)$,
                $A_{82}=\left(
        \begin{array}{ccccccccc }
2&0&1&2&0&1&2&0&1\\0&0&0&2&2&2&1&1&1\\1&0&2&2&1&0&0&2&1\\0&1&2&1&2&0&2&0&1\\
        \end{array}
        \right)$,\\
                $A_{83}=\left(
        \begin{array}{ccccccccc }
0&1&2&0&1&2&0&1&2\\1&1&1&0&0&0&2&2&2\\0&2&1&1&0&2&2&1&0\\0&1&2&1&2&0&2&0&1\\
        \end{array}
        \right)$,
                $A_{84}=\left(
        \begin{array}{ccccccccc }
0&1&2&0&1&2&0&1&2\\2&2&2&1&1&1&0&0&0\\1&0&2&2&1&0&0&2&1\\1&2&0&2&0&1&0&1&2\\
        \end{array}
        \right)$,
                $A_{85}=\left(
        \begin{array}{ccccccccc }
0&1&2&0&1&2&0&1&2\\0&0&0&2&2&2&1&1&1\\2&1&0&0&2&1&1&0&2\\2&0&1&0&1&2&1&2&0\\
        \end{array}
        \right)$,\\
                $A_{86}=\left(
        \begin{array}{ccccccccc }
1&2&0&1&2&0&1&2&0\\1&1&1&0&0&0&2&2&2\\1&0&2&2&1&0&0&2&1\\2&0&1&0&1&2&1&2&0\\
        \end{array}
        \right)$,
                $A_{87}=\left(
        \begin{array}{ccccccccc }
1&2&0&1&2&0&1&2&0\\2&2&2&1&1&1&0&0&0\\2&1&0&0&2&1&1&0&2\\0&1&2&1&2&0&2&0&1\\
        \end{array}
        \right)$,
                $A_{88}=\left(
        \begin{array}{ccccccccc }
1&2&0&1&2&0&1&2&0\\0&0&0&2&2&2&1&1&1\\0&2&1&1&0&2&2&1&0\\1&2&0&2&0&1&0&1&2\\
        \end{array}
        \right). $
\end{center}}}

\vskip5pt
 \noindent
        For any $i\in I_9$, $X+H_i$ and $Y+H_i^{*}$ both run  over $\mathbb{Z}_3^2$ when $X$ and $Y$ both run over $\mathbb{Z}_3^2$.
        Since $E$ is nonsingular and $E_1, E_2, E_1+E_2, E_1-E_2\in M_{4\times 2}^{(2)}(\mathbb{Z}_3)$,  $\{A_{i0},A_{i1},\cdots, A_{i8}\}$ is an SDLOA$(2,4,3)$ by the proof of Lemma \ref{OA1}. Thus $C_i$ forms an SDLOA$(9; 2, 4, 3)$ $\{A_{i0},A_{i1},\cdots, A_{i8}\}$ for each $i\in I_9$.
   One can check that the 9 SDLOA$(2,4,3)$s have the following properties:

 \noindent (1)  for fixed $k\in I_9$,   the $A_{0k}, A_{1k}, \cdots,A_{8k}$ forms a LOA$(2,4,3)$;

  \noindent   (2)  for fixed $l\in I_9$,   the $l$-th column of all $A_{ik}(i,k\in I_9)$  forms a LOA$(2,4,3)$;

  \noindent   (3) all the $k$-th column of $A_{ik}$, $i,k\in I_9$ forms a LOA$(2,4,3)$, and the $(8-k)$-th column of $A_{ik}$, $i,k\in I_9$ forms a LOA$(2,4,3)$.

  For any $C_i$, replacing  the arbitrary $(k,l)$-entry $(u,v,x,y)$ of $C_i$  by  $u+3v+9x+27y$, we get a matrix $\hat{C}_i$.
    By the proof of Construction \ref{main2},  $\hat{C}_0$, $\hat{C}_1$, $\cdots$ ,$\hat{C}_{8}$ are MS$(9,2)$s. We list $\hat{C}_0,\hat{C}_1,\cdots,\hat{C}_8$ as follows.

{{\renewcommand\arraystretch{0.8}
\setlength{\arraycolsep}{0.7pt}
\footnotesize
\begin{center}
$\hat{C}_0=\left(
        \begin{array}{ccccccccc }
46&65&0&61&26&42&13&32&75\\
58&23&39&10&29&72&52&71&6\\
16&35&78&49&68&3&55&20&36\\
2&45&64&44&60&25&77&12&31\\
41&57&22&74&9&28&8&51&70\\
80&15&34&5&48&67&38&54&19\\
63&1&47&24&43&62&30&76&14\\
21&40&59&27&73&11&69&7&53\\
33&79&17&66&4&50&18&37&56
        \end{array}
        \right)$,
        $\hat{C}_1=\left(
        \begin{array}{ccccccccc }
23&39&58&29&72&10&71&6&52\\
35&78&16&68&3&49&20&36&55\\
65&0&46&26&42&61&32&75&13\\
57&22&41&9&28&74&51&70&8\\
15&34&80&48&67&5&54&19&38\\
45&64&2&60&25&44&12&31&77\\
40&59&21&73&11&27&7&53&69\\
79&17&33&4&50&66&37&56&18\\
1&47&63&43&62&24&76&14&30
        \end{array}
        \right)$,
                $\hat{C}_2=\left(
        \begin{array}{ccccccccc }
78&16&35&3&49&68&36&55&20\\
0&46&65&42&61&26&75&13&32\\
39&58&23&72&10&29&6&52&71\\
34&80&15&67&5&48&19&38&54\\
64&2&45&25&44&60&31&77&12\\
22&41&57&28&74&9&70&8&51\\
17&33&79&50&66&4&56&18&37\\
47&63&1&62&24&43&14&30&76\\
59&21&40&11&27&73&53&69&7\\
        \end{array}
        \right)$,\\
                $\hat{C}_3=\left(
        \begin{array}{ccccccccc }
77&12&31&2&45&64&44&60&25\\
8&51&70&41&57&22&74&9&28\\
38&54&19&80&15&34&5&48&67\\
30&76&14&63&1&47&24&43&62\\
69&7&53&21&40&59&27&73&11\\
18&37&56&33&79&17&66&4&50\\
13&32&75&46&65&0&61&26&42\\
52&71&6&58&23&39&10&29&72\\
55&20&36&16&35&78&49&68&3\\
        \end{array}
        \right)$,
                $\hat{C}_4=\left(
        \begin{array}{ccccccccc }
51&70&8&57&22&41&9&28&74\\
54&19&38&15&34&80&48&67&5\\
12&31&77&45&64&2&60&25&44\\
7&53&69&40&59&21&73&11&27\\
37&56&18&79&17&33&4&50&66\\
76&14&30&1&47&63&43&62&24\\
71&6&52&23&39&58&29&72&10\\
20&36&55&35&78&16&68&3&49\\
32&75&13&65&0&46&26&42&61
        \end{array}
        \right)$,
                $\hat{C}_5=\left(
        \begin{array}{ccccccccc }
19&38&54&34&80&15&67&5&48\\
31&77&12&64&2&45&25&44&60\\
70&8&51&22&41&57&28&74&9\\
56&18&37&17&33&79&50&66&4\\
14&30&76&47&63&1&62&24&43\\
53&69&7&59&21&40&11&27&73\\
36&55&20&78&16&35&3&49&68\\
75&13&32&0&46&65&42&61&26\\
6&52&71&39&58&23&72&10&29
        \end{array}
        \right)$,\\
                $\hat{C}_6=\left(
        \begin{array}{ccccccccc }
24&43&62&30&76&14&63&1&47\\
27&73&11&69&7&53&21&40&59\\
66&4&50&18&37&56&33&79&17\\
61&26&42&13&32&75&46&65&0\\
10&29&72&52&71&6&58&23&39\\
49&68&3&55&20&36&16&35&78\\
44&60&25&77&12&31&2&45&64\\
74&9&28&8&51&70&41&57&22\\
5&48&67&38&54&19&80&15&34
        \end{array}
        \right)$,
                $\hat{C}_7=\left(
        \begin{array}{ccccccccc }
73&11&27&7&53&69&40&59&21\\
4&50&66&37&56&18&79&17&33\\
43&62&24&76&14&30&1&47&63\\
29&72&10&71&6&52&23&39&58\\
68&3&49&20&36&55&35&78&16\\
26&42&61&32&75&13&65&0&46\\
9&28&74&51&70&8&57&22&41\\
48&67&5&54&19&38&15&34&80\\
60&25&44&12&31&77&45&64&2
        \end{array}
        \right)$,
                $\hat{C}_8=\left(
        \begin{array}{ccccccccc }
50&66&4&56&18&37&17&33&79\\
62&24&43&14&30&76&47&63&1\\
11&27&73&53&69&7&59&21&40\\
3&49&68&36&55&20&78&16&35\\
42&61&26&75&13&32&0&46&65\\
72&10&29&6&52&71&39&58&23\\
67&5&48&19&38&54&34&80&15\\
25&44&60&31&77&12&64&2&45\\
28&74&9&70&8&51&22&41&57
        \end{array}
        \right).$
\end{center}}}

 \noindent By the property (1),  for fixed $k\in I_9$,  the union of the columns  of  $A_{0k}, A_{1k}, \cdots,A_{8k}$ is exactly the set  $F_3^4$.
 So the numbers of all the $k$-th rows of   $\hat{C}_0$, $\hat{C}_1$, $\cdots$ ,$\hat{C}_{8}$ forms the set $I_{81}$  exactly, which is independent of the choice of $k$.  So the condition  $R_1$ in the definition of CMS  is  satisfied. In the same way, by the  properties (2) and (3) the conditions $R_2$ and $R_3$ are satisfied, respectively. Thus $\{\hat{C}_0, \hat{C}_1, \cdots, \hat{C}_8\}$ is a 9-CMS$(9,2)$.
   \end{proof}

We now give a sufficient condition of $q^t$-CMS$(q^t,t)$ based on LOA while keeping the idea of 9-CMS$(9,2)$ in mind.

 \begin{lem} \label{LOA-CMS}
 If there exist $q^t$ SDLOA$(q^t; t, 2t, q)$s over $F_q$,  $C_H=(C_{X,Y}^{(H)})$, $X, Y, H\in F_q^t$, such that
\begin{center}
$\begin{aligned}
  &R_X=\{C^{(H)}_{X,Y}| Y, H\in F_q^t\}, X\in F_q^t,\\
  &T_Y=\{C^{(H)}_{X,Y}| X, H\in F_q^t\}, Y\in F_q^t,\\
  &U=\{C^{(H)}_{X,X}| X,H\in F_q^t\},\\
  &U'=\{C^{(H)}_{X,\widetilde{X}-X}| X,H\in F_q^t\}
    \end{aligned}$
\end{center}
are $2q^t+2$ LOA$(q^t; t, 2t, q)$s, then there exists a  $q^t$-CMS$(q^t,t)$.
 \end{lem}
 \begin{proof}
 Let  $\hat{C}_H=(\hat{C}_{X,Y}^{(H)})$, where
  $$\hat{C}_{X,Y}^{(H)}=j_0+j_1 q+\cdots+j_{2t-1}q^{2t-1}=(1,q,\cdots,q^{2t-1}) (j_0,j_1,\cdots,j_{2t-1})^T$$
 if $C_{X,Y}^{(H)}=(\xi_{j_0}, \xi_{j_1}, \cdots, \xi_{j_{2t-1}})^T,$  $X, Y, H\in F_q^t$. By the proof of Lemma \ref{main2}, $\hat{C}_H, H\in F_q^t$ are $q^t$ MS$(q^t,t)$s. We shall prove that  $\{\hat{C}_H| H\in F_q^t\}$  is a $q^t$-CMS$(q^t,t)$.

 In fact, for fixed $X\in F_q^t$, since $R_X$  forms  an LOA$(q^t; t, 2t,q)$,  we have $$\{\hat{C}^{(H)}_{X,Y}| Y, H\in F_q^t\}=I_{q^{2t}}.$$
 Thus
$  \sum\limits_{H\in F_q^t}\sum\limits_{Y\in F_q^t}(\hat{C}^{(H)}_{X,Y})^{t+1}=\sum\limits_{d=0}^{q^{2t}-1}d^{t+1}, X\in F_q^t.$
 But $\sum\limits_{d=0}^{q^{2t}-1}d^{t+1}=q^tS_{t+1}(q^t)$.  So we have
 \begin{center}
$  \sum\limits_{H\in F_q^t}\sum\limits_{Y\in F_q^t}(\hat{C}^{(H)}_{X,Y})^{t+1} =q^tS_{t+1}(q^t), X\in F_q^t.$
\end{center}
 Similarly, for fixed $Y$,   $T_Y$ forms an  LOA$(q^t; t, 2t, q)$. So we have
 \vskip5pt
$ \  \ \ \ \ \ \ \ \  \ \ \ \ \ \ \  \ \ \ \ \  \ \ \ \ \ \ \ \sum\limits_{H\in F_q^t}\sum\limits_{X\in F_q^t}(\hat{C}^{(H)}_{X,Y})^{t+1} =q^tS_{t+1}(q^t), Y\in F_q^t;\ \ \ \ \ \  \ \ \ \ \  \ \ \ \   \ \   \ \ \ \ \ $
 \vskip5pt
\noindent In the same way, since $U$ and $U'$ both forms  LOA$(q^t; t, 2t, q)$s. So
 \vskip5pt
  $  \  \ \ \ \  \ \ \  \ \ \ \  \ \ \ \ \ \sum\limits_{H\in F_q^t}\sum\limits_{X\in F_q^t}(\hat{C}^{(H)}_{X,X})^{t+1}=\sum\limits_{H\in F_q^t}\sum\limits_{X\in F_q^t}(\hat{C}^{(H)}_{X,\widetilde{X}-X})^{t+1}=q^tS_{t+1}(q^t).$
   \vskip5pt
\noindent So, $\{\hat{C}_H| H\in F_q^t\}$  is a $q^t$-CMS$(q^t,t)$.
  \end{proof}

Now the  technique of matrix is used to construct a $q^t$-CMS$(q^t,t)$.
 \begin{lem} \label{OA2}
Let $t\geq2$ and $d\in F_q$.  If  there exists a  matrix $E=(e_{i,j})_{2t\times 2t}=(E_1, E_2)$ over $F_q$ such that\\
\emph{(1)} $E_1$, $E_2$, $E_1+E_2$, $E_1-E_2$   are  all in  $M_{2t\times t}^{(t)}(F_q)$,\\
\emph{(2)} $E$, $(E_1,E_1+dE_2)$, $(E_2,E_1+dE_2)$, $(E_1+E_2, E_1+dE_2)$, $(E_1-E_2, E_1+dE_2)$ are  nonsingular,
 then  there exists  a $q^t$-CMS$(q^t,t)$.
 \end{lem}
 \begin{proof}
  Let $C=(C_{X,Y})$, where
$C_{X,Y}=(E_1, E_2)
\left(
\begin{array}{l}
X\\
Y\\
\end{array}
\right),
X,Y\in F_q^t.$
By the proof of Lemma \ref{OA1}, $C$ forms an SDLOA$(q^t; t, 2t, q)$ over $F_q$.
 We shall construct a $q^t$-CMS$(q^t,t)$ by making use of the coordinate translations of $C$.

 Let $H\in F_q^t$ and $H^*=dH$. Let
 \begin{center}
$C_{H}=(C^{(H)}_{X,Y})$, $H\in F_q^t$,
\end{center}
where
\begin{center}
$C^{(H)}_{X,Y}=C_{X+H, Y+H^*}=(E_1, E_2)
\left(
\begin{array}{l}
X+H\\
Y+H^*
\end{array}
\right), \
X,Y\in F_q^t.$
\end{center}
i.e.,
\begin{center}
$C^{(H)}_{X,Y}=E_1X+E_2Y+(E_1+dE_2)H,\ \
X,Y\in F_q^t.$
\end{center}
Let $\hat{C}_H$ be the same as defined in Lemma \ref{LOA-CMS}. We shall show that  $\{\hat{C}_H| H\in F_q^t\}$  is a $q^t$-CMS$(q^t,t)$.

For fixed $H\in F_q^t$, let $K_H=(E_1, E_2)
\left(
\begin{array}{l}
H\\
H^*
\end{array}
\right)$,
then
\begin{center}
$C^{(H)}_{X,Y}=(E_1, E_2)
\left(
\begin{array}{l}
X\\
Y
\end{array}
\right)+K_H,\ \ X,Y\in F_q^t,$
\end{center}
which means that $C_H$ is obtained by adding $K_H$ to each column of the 3-dimensional array $C$. Since $C$ is an SDLOA$(q^t; t, 2t, q)$, it is obvious that
$C_H$ is also an SDLOA$(q^t; t, 2t, q)$.
Let $R_X(X\in F_q^t), T_Y(Y\in F_q^t), U, U'$ be the same as in Lemma \ref{LOA-CMS}.
By Lemma \ref{LOA-CMS}, it remains to prove that $R_X (X\in F_{q^t}), T_Y (Y\in F_{q^t}), U, U'$ are all LOA$(q^t; t, 2t, q)$s.
We have
\begin{center}
$\begin{aligned}
  R_X&=\{E_1X+E_2Y+(E_1+dE_2)H| Y, H\in F_q^t\}\\
  &=\{E_1X+(E_2,E_1+dE_2)
  \left(
\begin{array}{l}
Y\\
H
\end{array}
\right)
  | Y, H\in F_q^t\}, \ \ X\in F_q^t\\
 \end{aligned} $
\end{center}
Since $E_2\in M_{2t\times t}^{(t)}(F_q)$ and $(E_2,E_1+dE_2)$ is nonsingular, $R_X$ forms an  LOA$(q^t; t, 2t, q)$ for each $X\in F_q^t$.
Similarly, since $E_1\in M_{2t\times t}^{(t)}(F_q)$ and $(E_1,E_1+dE_2)$ is nonsingular, $T_Y$  forms an   LOA$(q^t; t, 2t, q)$ for each $Y\in F_q^t$.

Now we consider $U, U'$, by Lemma \ref{LOA1} we have
\begin{center}
$\begin{aligned}
  U&=\{E_1X+E_2X+(E_1+dE_2)H| X, H\in F_q^t\}\\
  &=\{(E_1+E_2,E_1+dE_2)
  \left(
\begin{array}{l}
X\\
H
\end{array}
\right)
  | X, H\in F_q^t\}.\\
 \end{aligned} $
\end{center}
and
\begin{center}
$\begin{aligned}
  U'   &=\{E_1X+E_2( \widetilde{X}-X)+(E_1+dE_2)H|  X, H\in F_q^t\}\\
&=\{(E_1-E_2,  E_1+dE_2)
 \left(\begin{array}{l}
X\\
H
\end{array}
\right)+E_2\widetilde{X} | X, H\in F_q^t\}
   \end{aligned}$
\end{center}
 since $E_1+E_2,E_1-E_2 \in M_{2t\times t}^{(t)}(F_q)$ and $(E_1+E_2,E_1+dE_2)$, $(E_1-E_2, E_1+dE_2)$ are nonsingular, $U, U'$ are both LOA$(q^t; t, 2t, q)$s.
The proof is completed.
 \end{proof}

\begin{lem}\label{generate}
Let $q$ be a prime power and $q\ge 4$. If  $ M_{2t\times t}^{(t)}(F_q)\neq \emptyset$, then there exists a  $q^t$-CMS$(q^t,t)$.
\end{lem}
\begin{proof}
  Let $x$ be a primitive element in  $F_q$.
  Since $q\geq4$, which indicates that $x,x^2\not\in\{ 0, 1,-1\}$ and $x\neq x^2$. Suppose that $E_{11}, E_{21}$ are two $t\times t$ matrices such that
  $E_1=\left(
        \begin{subarray}{}
E_{11} \\
E_{21}
        \end{subarray}
        \right)\in M_{2t\times t}^{(t)}(F_q)$.
 Let
  \begin{center}
      $E=(E_{1},\  E_{2})=\left(
        \begin{array}{l}
E_{11} \ \ E_{12} \\
E_{21} \ \ E_{22}
        \end{array}
        \right)
        =\left(
        \begin{array}{l}
E_{11} \ \ xE_{11} \\
E_{21} \ \ x^2E_{21}
        \end{array}
        \right),$
 \end{center}
and $d=x$. It is readily checked that $E$ satisfies the conditions in Lemma \ref{OA2}.
Therefore there exists a  $q^t$-CMS$(q^t,t)$ by Lemma \ref{OA2}.
   \end{proof}

 \begin{thm}\label{maincms}
There exists a $q^t$-CMS$(q^t,t)$ for any prime power $q\geq 2t-1$ with $t\geq 2$.
    \end{thm}
  \begin{proof}
For $q=3$ and $t=2$, there exist  $9$-CMS$(9,2)$ by Lemma \ref{cms9}. For $q\geq max\{2t-1,4\} $ and $t\geq 2$, let $x$ be a primitive element of $F_q$ and
 {\renewcommand\arraystretch{0.7}
\setlength{\arraycolsep}{1.8pt}
 \footnotesize
\begin{center}
$E=
        \left(
        \begin{array}{cccccc}
1& 0& 0&\cdots &0\\
0& 0&0& \cdots&1\\
1& x& x^2&\cdots &x^{t-1}\\
1& x^2&  x^4&\cdots&x^{2(t-1)}\\
\vdots&\vdots&\vdots&\vdots&\vdots\\
1& x^{2t-2}&x^{2(2t-2)}&\cdots& x^{(t-1)(2t-2)}\\
        \end{array}
        \right). $
\end{center}}
\noindent
It is readily checked that any $t$ rows of $E$ are linearly independent, so $E\in M_{2t\times t}^{(t)}(F_q)$. Thus, there exists a  $q^t$-CMS$(q^t,t)$  by Lemma  \ref{generate}.
   \end{proof}

 \section{Main results}

\noindent {\bf Theorem \ref{q2t-1}}
There exists an MS$(q^{2t-1},t)$ for any prime power $q\geq 2t-1$ with $t\geq3$.
\begin{proof}
 For any prime power $q\geq 2t-1$ with $t\geq3$,   there exists an MS$(q^t,t)$ by Lemma \ref{zhang} and exists an  $q^{t-1}$-CMS$(q^{t-1},t-1)$   by Theorem \ref{maincms}. So  there exists an MS$(q^{2t-1},t)$ by Construction \ref{conmscms}.
\end{proof}

\noindent {\bf Remark}   The MS$(q^{2t-1},2t-1)$s coming from Theorem \ref{q2t-1} with prime power $q<20$ are listed as follows.
{\renewcommand\arraystretch{0.8}
\setlength{\arraycolsep}{1.8pt}
 \footnotesize
\begin{center}
$\begin{array}{cccccccccccccc}
 &\color{red}(5^5,3)&\color{red}(7^5,3)&\color{red}(8^5,3)&(9^5,3)&(11^5,3)&(13^5,3)&(16^5,3)&(17^5,3)&(19^5,3)\\
&& \color{red}(7^7,4)&\color{red}(8^7,4)&\color{red}(9^7,4)&\color{red}(11^7,4)&(13^7,4)&(16^7,4)&(17^7,4)&(19^7,4)\\
 &&&  &\color{red}(9^9,5)&\color{red}(11^9,5)&\color{red}(13^9,5)&\color{red}(16^9,5)&(17^9,5)&(19^9,5)\\
 &&&&&\color{red}(11^{11},6)&\color{red}(13^{11},6)&\color{red}(16^{11},6)&\color{red}(17^{11},6)&\color{red}(19^{11},6)\\
 &&&&&& \color{red}(13^{13},7)&\color{red}(16^{13},7)&\color{red}(17^{13},7)&\color{red}(19^{13},7)\\
 &&& &&&&\color{red}(16^{15},8)& \color{red} (17^{15},8)&\color{red}(19^{15},8)\\
  & &&&&&& &\color{red}(17^{17},9)& \color{red}(19^{17},9)
        \end{array}$
\end{center}}
\noindent We should point out that those  marked in red  can't be obtained  from Lemma \ref{zhang} directly. Further, we have the following new family of multimagic squares.

\begin{cor}\label{thm-q2t-1}
There exists an MS$(q^{2t-1},t)$ for any prime power $2t-1 \leq q< 4t-3$ with $t\geq3$.
\end{cor}
\begin{proof}
For a $(2t-1)$-multimagic square, Lemma \ref{zhang} provides an  MS$(q^{2(2t-1)-1},2t-1)$ for any prime power $q\geq 2(2t-1)-1$ with $2t-1\geq2$, i.e.,
an MS$(q^{4t-3},2t-1)$ for any prime power $q\geq 4t-3$ with $t\geq2$. But Theorem \ref{q2t-1}   provides an  MS$(q^{2t-1},2t-1)$ for any prime power $q\geq 2t-1$ with $t\geq3$. So the MS$(q^{2t-1},t)$s for   prime power $2t-1 \leq q< 4t-3$ with $t\geq3$ are covered by Theorem \ref{q2t-1} but are not covered by Lemma \ref{zhang}.
\end{proof}

For fixed $q$ and fixed $t$, those multimagic squares come from Lemma \ref{zhang} and Theorem  \ref{q2t-1} are both single case, but the following provided a family ones.
\begin{cor}\label{main-mst}
There exists an MS$(q^{m},t)$ for  any prime power  $q$  with $q\geq 4t-5$ and  $m\geq t\geq3$.
\end{cor}
\begin{proof}
Let $t\geq3$ and let $m=t+k, k\geq0$. For $0\leq k\leq t-2$ and $q\geq 2(t+k)-1$, by Lemma \ref{zhang} there exists an MS$(q^{t+k},t+k)$  and hence there exists an MS$(q^{t+k},t)$. For $k=t-1$, there exists an MS$(q^{t+k},t)$  when $q\geq2t-1$ by Theorem \ref{q2t-1}. For any $m\ge 2t$, we can write $m=\sum_{i=0}^{t-1}(t+i)n_i$, where $n_0,n_1,\dots,n_{t-1}$ are nonnegative integers. By Construction \ref{conproduct}, we get an MS$(q^{m},t)$. The proof is completed.
\end{proof}

\noindent {\bf Open problem} Find more MS$(q^{t},t)$  and  MS$(q^{2t-1},t)$ for prime power $q<2t-1$  with $t\geq3$.

\vskip10pt


\end{document}